\documentclass[12pt,reqno]{amsart}

\usepackage{mathrsfs}
\usepackage[textwidth=16cm,hmarginratio=1:1]{geometry}
\usepackage[colorlinks=true,linkcolor=blue,citecolor=blue,urlcolor=blue]{hyperref}

\theoremstyle{plain}
\newtheorem{theorem}{Theorem}[section]
\newtheorem{proposition}[theorem]{Proposition}
\newtheorem{lemma}[theorem]{Lemma}
\newtheorem{corollary}[theorem]{Corollary}
\theoremstyle{remark}
\newtheorem{remark}[theorem]{Remark}
\newtheorem{example}[theorem]{Example}
\numberwithin{equation}{section}

\newcommand{\ZZ}{\mathbb{Z}}
\newcommand{\fg}{\mathfrak{g}}
\newcommand{\bw}{\mathbf{w}}
\newcommand{\bv}{\mathbf{v}}
\newcommand{\ba}{\mathbf{a}}
\newcommand{\bdelta}{\boldsymbol{\delta}}
\DeclareMathOperator{\wt}{wt}
\DeclareMathOperator{\supp}{supp}

\title{Lusztig Parameters and Factorizations of Determinantial Modules}
\author{Yingjin Bi}
\address{Department of Mathematics, Harbin Engineering University; Research Institute for Mathematical Sciences, Kyoto University}
\email{yingjinbi@mail.bnu.edu.cn}
\date{}

\begin{document}
\raggedbottom

\begin{abstract}
We give an explicit formula for the Lusztig parameters of
determinantial modules over symmetric quiver Hecke algebras.
The formula applies to every Bruhat pair and every fixed reduced
expression of the upper Weyl group element.
In finite types $A$, $D$, and $E$, we identify the parameters of
the modules in M\'enard's seed using Demazure weaves.
We then factor prefix determinantial modules into these seed
modules, with multiplicities given by Lusztig cycles.
We also give a criterion for recognizing monomials in this fixed
seed and an explicit formula recovering their multiplicities
from Lusztig parameters.
\end{abstract}

\maketitle

\section{Introduction}

Quantum minors play an important role in the quantum cluster
structures on quantum unipotent coordinate rings
\cite{geiss2013cluster}. Under quiver Hecke categorification,
the quantum minor $D(w\Lambda,v\Lambda)$ corresponds to the
self-dual determinantial module $M(w\Lambda,v\Lambda)$, where
$v\leq w$ and $\Lambda$ is a dominant integral weight
\cite{kkop2018strata}.
For a reduced expression $\bw$ of $w$, each simple module in
the quantum unipotent category $\mathscr C_w$ is determined,
up to grading shift, by its $\bw$-Lusztig parameter.
Computing this parameter therefore provides a way to identify
determinantial modules and study their convolution factorizations.

The familiar description in terms of $\mathbf i$-boxes treats
the special case
\[
 w=vu,\qquad \ell(w)=\ell(v)+\ell(u),
\]
with a compatible reduced expression $\bw=\bv\mathbf u$.
For fundamental weights, the corresponding nontrivial
determinantial modules are associated with $\mathbf i$-boxes,
and their Lusztig parameters record the occurrences of the
relevant simple root in the $\mathbf u$-block; see
\cite[Proposition~4.6 and Corollary~4.7]{kkop2018strata}
for the cuspidal decomposition and
\cite[Sections~3.2 and~4.4]{kkop2024affineII}
for the $\mathbf i$-box formulation.
This description does not directly give the parameters for
an arbitrary pair $v\leq w$ and a prescribed reduced expression
of $w$.

Our first result gives such a formula in all symmetric
Kac--Moody types.
Fix $\bw=(i_1,\ldots,i_r)$, and choose the leftmost reduced
subexpression of $v$ in $\bw$.
It determines a vector $\bdelta_{\bw,v}(\Lambda)$: each selected
coordinate is zero, while an unselected coordinate is obtained
by pairing $\Lambda$ with the simple coroot at that position,
transformed by the selected suffix.
Theorem~\ref{thm:lusztig-determinantial} proves that
\[
 \ba^{\bw}\bigl(M(w\Lambda,v\Lambda)\bigr)
 =\bdelta_{\bw,v}(\Lambda).
\]
The precise coordinates are given in~\eqref{eq:delta-definition}.
The proof follows the recursive construction of Lusztig
parameters using crystal operators and Saito reflections.
Its key step is the simultaneous-descent identity in
Lemma~\ref{lem:determinantial-simultaneous-descent}.

The factorization problem is motivated by cluster structures
on open Richardson varieties.
Leclerc~\cite{leclerc2016cluster} constructed a cluster
subalgebra of the coordinate ring of an open Richardson variety,
with initial variables given by irreducible factors of certain
generalized minors.
M\'enard~\cite{menard2022seeds} subsequently gave an explicit
algorithm for computing the associated seeds.
At the categorical level, the author proved that the factors of
the prefix determinantial modules
$M(w_{\leq k}\varpi_{i_k},v_{\leq k}\varpi_{i_k})$
are the cluster-variable modules of the corresponding seed
\cite[Theorems~5.11 and~5.15]{bi2026monoidalrichardson}.
Kashiwara--Kim--Oh--Park~\cite{kashiwara2026seed}
constructed these modules using the operators $\mathcal K_i$
and proved the existence of such factorizations, together with
a unitriangularity property of their multiplicities.
Here our aim is to compute the multiplicities explicitly.

In type~$A$, Serhiyenko--Sherman-Bennett~\cite{serhiyenko2024leclerc}
related the generalized minors in Leclerc's construction to
Ingermanson's chamber minors~\cite{ingermanson2019cluster}
through the Richardson twist.
Under the identification of the open Richardson variety with
the corresponding braid variety, these Ingermanson minors
coincide with the braid-variety chamber minors
\cite[Proposition~10.5]{galashin2022braidI}.
Consequently, the chamber-minor factorizations of
Galashin--Lam--Sherman-Bennett~\cite{galashin2026braidII}
and their type~$A$ counterpart~\cite{galashin2022braidI}
can be transported through the twist to obtain factorizations
of the corresponding generalized minors in type~$A$.

Although chamber-minor factorizations are available in general
type~\cite{galashin2026braidII}, an analogous identification
between the generalized minors in Leclerc's construction and
the braid-variety chamber minors has not, to our knowledge,
been established in types~$D$ and~$E$.
Thus the existing classical factorizations do not directly
yield the desired factorizations of generalized minors in
these types.
Our second goal is to establish explicit quantum and
categorical factorizations in all finite types $A$, $D$,
and~$E$, without relying on such an identification.

For this purpose, choose a reduced word $\mathbf h$ for
$w_0v^{-1}$ and use the right inductive Demazure weave for
$\mathbf h\bw$.
Its trivalent vertices index the cluster-variable modules
$M_u$ of M\'enard's seed $\mathbf s(v,\bw)$
\cite{casals2025cluster}.
The weave vectors and their local mutation rules were developed
in~\cite{bi2026monoidal}. In particular, Theorem~6.17 of that
paper gives the coroot identity used below, while Theorem~7.12
computes the parameters of quantum grid minors in bosonic
extension algebras.
In the present paper, we identify the weave vectors with the
ordinary $\bw$-Lusztig parameters of the modules $M_u$
(Proposition~\ref{prop:weave-parameters}) and combine this
identification with our determinantial parameter formula.
This yields, for every $\Lambda\in P_+$,
\[
 M(w_{\leq k}\Lambda,v_{\leq k}\Lambda)
 \simeq
 \bigodot_{q_u\leq k}
 M_u^{\odot\langle\gamma_{u,k},\Lambda\rangle}.
\]
Here $v_{\leq k}$ is the selected prefix, $q_u$ is the position
creating $u$, and $\gamma_{u,k}$ is the coroot obtained from
the Lusztig cycle of $u$ on the slice after step $k$.
Thus the factorization multiplicities of the usual
determinantial modules are given by explicit nonnegative
coroot pairings.
Passing to Grothendieck classes gives the corresponding
factorizations of quantum minors.

Finally, the triangular recursion for the weave vectors gives
a linear criterion for recognizing monomials in this fixed
seed and a formula for recovering their unique multiplicities
directly from Lusztig parameters
(Proposition~\ref{prop:seed-monomial}).
Section~\ref{sec:preliminaries} recalls the required categorical
and crystal facts.
Section~\ref{sec:parameters} proves the determinantial parameter
formula, and Section~\ref{sec:weaves} develops its finite-type
consequences and illustrates the computations with examples.\\

\noindent
\textbf{Use of AI tools.} The author used ChatGPT (OpenAI) to assist with language editing, the organization and presentation of the manuscript, and the formulation, revision, and logical checking of parts of the mathematical arguments. The author takes full responsibility for all mathematical statements and proofs, the accuracy of the references, and the final content of the manuscript.

\section{Preliminaries}\label{sec:preliminaries}

We recall the conventions and standard facts used in the proofs.
Unless otherwise stated, the Cartan datum is symmetric, but need not
be of finite type.

\subsection{Quiver Hecke modules and quantum unipotent categories}

Let $\fg$ be a symmetric Kac--Moody algebra with index set $I$,
Weyl group $W$, simple roots $\alpha_i$, simple coroots
$h_i=\alpha_i^\vee$, and fundamental weights $\varpi_i$.
Set $Q_+=\sum_{i\in I}\ZZ_{\geq0}\alpha_i$ and $Q_-=-Q_+$.
We write $P_+$ for the dominant integral weights and $Q^\vee$ for
the coroot lattice. The pairing $\langle h,\lambda\rangle$ has the
coroot in its first argument. The symbols $\leq$ and $\ell$ denote
Bruhat order and the length function on $W$, respectively.

Let $R$ be a symmetric quiver Hecke algebra over an algebraically
closed field of characteristic zero. We write $R$-$\mathrm{gmod}$
for the direct sum, over $\nu\in Q_+$, of the categories of
finite-dimensional graded $R(\nu)$-modules. Convolution, denoted
by $\circ$, is the induction product. For $X\in R(\nu)$-$\mathrm{gmod}$,
our weight convention is $\wt(X)=-\nu$. A grading shift does not
change the weight or any of the parameters considered below.

For $X\in R(\nu)$-$\mathrm{gmod}$, define its restriction supports by
\[
 \begin{split}
 \mathcal W(X)&=\{\gamma\in Q_+\cap(\nu-Q_+):
       \operatorname{Res}_{\gamma,\nu-\gamma}X\neq0\},\\
 \mathcal W^*(X)&=\{\gamma\in Q_+\cap(\nu-Q_+):
       \operatorname{Res}_{\nu-\gamma,\gamma}X\neq0\}.
 \end{split}
\]
For $w\in W$, the quantum unipotent category $\mathscr C_w$ is
the full subcategory defined by
\[
 X\in\mathscr C_w\quad\Longleftrightarrow\quad
 \mathcal W(X)\subseteq Q_+\cap wQ_-.
\]
This condition is imposed on each $R(\nu)$-module summand of $X$.
The category is stable under subquotients, extensions, convolution,
and grading shifts; see \cite[Section~2]{kkop2018strata}.

Every simple module has a unique self-dual grading shift, where
duality means graded linear duality with respect to the standard
anti-involution of $R$. A simple module $X$ is \emph{real} if
$X\circ X$ is simple. Two simple modules \emph{strongly commute}
if their convolution is simple. For a pairwise strongly commuting
family of self-dual real simple modules, $\bigodot$ denotes the
convolution with the unique grading shift making it self-dual.
Empty products and zeroth powers mean the tensor unit.

\subsection{Crystal operators and Saito reflection}

We identify simple modules that differ by a grading shift.
The crystal operators $\widetilde E_i,\widetilde F_i$ act on the
left, and $\widetilde E_i^*,\widetilde F_i^*$ act on the right.
In particular, for a simple module $X$,
\[
 \begin{aligned}
 \varepsilon_i(X)&=\max\{n\geq0:n\alpha_i\in\mathcal W(X)\},\\
 \varepsilon_i^*(X)&=\max\{n\geq0:n\alpha_i\in\mathcal W^*(X)\}.
 \end{aligned}
\]
Set $\varphi_i(X)=\varepsilon_i(X)+\langle h_i,\wt(X)\rangle$.
Under the identification with $B(\infty)$, Saito reflection is
the bijection $T_i$ from $\varepsilon_i^*=0$ to $\varepsilon_i=0$
given by
\[
 T_iX=(\widetilde F_i^*)^{\varphi_i(X)}
                \widetilde E_i^{\,\varepsilon_i(X)}X.
\]
These conventions agree with \cite{saito1994pbw} and
\cite[Section~4]{kimura2012quantum}.

\begin{lemma}[Stabilized Saito reflection]\label{lem:saito-stabilized}
Let $X$ be a simple module with $\varepsilon_i^*(X)=0$.
For every integer $s\geq\varphi_i(X)$, one has
\begin{equation}\label{eq:saito-stabilized}
 T_iX=\widetilde E_i^{\,s-\langle h_i,\wt(X)\rangle}
               (\widetilde F_i^*)^sX.
\end{equation}
\end{lemma}
\begin{proof}
We work with crystal classes. Set
\[
 m=\varepsilon_i(X),\qquad
 d=\langle h_i,\wt(X)\rangle,\qquad
 p=\varphi_i(X)=m+d.
\]
For a crystal element $Y$, define
\[
 \kappa_i(Y)
 :=\varepsilon_i(Y)+\varepsilon_i^*(Y)
       +\langle h_i,\wt(Y)\rangle.
\]
The bicrystal relations give $\kappa_i(Y)\geq0$; see
\cite[Proposition~1.4]{tingley2016mv}, translated from the
$B(-\infty)$ convention there to our $B(\infty)$ convention.
Since $\varepsilon_i^*(X)=0$, we have
$\kappa_i(X)=p\geq0$.

We recall the relations needed below.
If $\kappa_i(Y)>0$, then
\[
 \varepsilon_i(\widetilde F_i^*Y)=\varepsilon_i(Y),
 \qquad
 \kappa_i(\widetilde F_i^*Y)=\kappa_i(Y)-1.
\]
If $\kappa_i(Y)>0$ and $\widetilde E_iY\neq0$, then
\[
 \widetilde E_i\widetilde F_i^*Y
 =\widetilde F_i^*\widetilde E_iY,
 \qquad
 \kappa_i(\widetilde E_iY)=\kappa_i(Y)+1.
\]
If $\kappa_i(Y)=0$, then
$\widetilde F_i^*Y=\widetilde F_iY$, and the resulting
element still has $\kappa_i=0$.

Put $Y_t=(\widetilde F_i^*)^tX$.
Induction on $t$ gives
\[
 \varepsilon_i(Y_t)=m,\qquad
 \kappa_i(Y_t)=p-t
 \qquad(0\leq t\leq p).
\]
For $t<p$, the value $\kappa_i(Y_t)$ is positive and
increases under each nonzero application of $\widetilde E_i$.
Thus the preceding commutation relation can be applied
repeatedly, yielding
\[
 \widetilde E_i^{\,m}Y_p
 =\widetilde E_i^{\,m}(\widetilde F_i^*)^pX
 =(\widetilde F_i^*)^p\widetilde E_i^{\,m}X
 =T_iX.
\]

Since $\kappa_i(Y_p)=0$, all subsequent applications of
$\widetilde F_i^*$ agree with those of $\widetilde F_i$.
Consequently, for every $s\geq p$,
\[
 (\widetilde F_i^*)^sX
 =\widetilde F_i^{\,s-p}Y_p.
\]
Finally, $s-d=m+s-p\geq0$, so the inverse-string identity
$\widetilde E_i^{\,a}\widetilde F_i^{\,a}Y=Y$ gives
\[
 \begin{aligned}
 \widetilde E_i^{\,s-d}(\widetilde F_i^*)^sX
 &=\widetilde E_i^{\,m+s-p}
       \widetilde F_i^{\,s-p}Y_p\\
 &=\widetilde E_i^{\,m}Y_p\\
 &=T_iX.
 \end{aligned}
\]
This proves~\eqref{eq:saito-stabilized}.
\end{proof}

\subsection{Determinantial modules}

For $\Lambda\in P_+$ and $y\leq x$, the determinantial module
$M(x\Lambda,y\Lambda)$ is the self-dual simple module whose
Grothendieck class is the unipotent quantum minor
$D(x\Lambda,y\Lambda)$. It is real, belongs to $\mathscr C_x$,
and has weight $x\Lambda-y\Lambda$; see
\cite[Propositions~4.1--4.2 and Theorem~4.5]{kkop2018strata}.
For $M=M(x\Lambda,y\Lambda)$, the restriction-support property
\cite[Lemma~4.4]{kkop2018strata} states that
\begin{equation}\label{eq:determinantial-support}
 \begin{split}
 \gamma\in\mathcal W(M)&\Longrightarrow
           x\Lambda+\gamma\in\wt V(\Lambda),\\
 \gamma\in\mathcal W^*(M)&\Longrightarrow
           y\Lambda-\gamma\in\wt V(\Lambda),
 \end{split}
\end{equation}
where $\wt V(\Lambda)$ denotes the weight set of the integrable
highest-weight module $V(\Lambda)$.

The following extremal-string identities follow from
\cite[Lemmas~1.4 and~1.7]{kkop2018strata}, together with the
identification of quantum minors with determinantial modules.
If $s_i y>y$ and $s_i y\leq x$, then, with
$n=\langle h_i,y\Lambda\rangle\geq0$,
\begin{equation}\label{eq:star-extremal-string}
 \varepsilon_i^*(M)=n,\qquad
 (\widetilde E_i^*)^nM\simeq M(x\Lambda,s_i y\Lambda).
\end{equation}
If $s_i x<x$ and $y\leq s_i x$, then, with
$a=-\langle h_i,x\Lambda\rangle\geq0$,
\begin{equation}\label{eq:unstarred-extremal-string}
 \varepsilon_i(M)=a,\qquad
 \widetilde E_i^{\,a}M\simeq M(s_i x\Lambda,y\Lambda).
\end{equation}
Isomorphisms involving crystal operators are understood up to
grading shift.

\subsection{Lusztig parameters}

Fix a reduced expression $\bw=(i_1,\ldots,i_r)$ of $w$, and set
$\rho_k=s_{i_1}\cdots s_{i_{k-1}}(\alpha_{i_k})$.
The PBW parametrization assigns to every simple module
$X\in\mathscr C_w$, modulo grading shift, a unique vector
$\ba^{\bw}(X)\in\ZZ_{\geq0}^r$, and gives a bijection between
these simple classes and $\ZZ_{\geq0}^r$.
If $r>0$, set $i=i_1$. The Saito--PBW construction gives
\[
 T_i(\widetilde E_i^*)^{\varepsilon_i^*(X)}X
 \in\mathscr C_{s_iw}.
\]
Thus the parameter for the shorter reduced word is defined,
and our convention is characterized by the recursion
\begin{equation}\label{eq:lusztig-saito-recursion}
 \begin{split}
 a_1&=\varepsilon_{i_1}^*(X),\\
 (a_2,\ldots,a_r)
 &=\ba^{(i_2,\ldots,i_r)}
       \left(T_{i_1}(\widetilde E_{i_1}^*)^{a_1}X\right),
 \qquad \ba^{\bw}(X)=(a_1,\ldots,a_r).
 \end{split}
\end{equation}
The empty word parametrizes the tensor unit. In this convention,
$\wt(X)=-\sum_{k=1}^r a_k\rho_k$.
For $w_{\leq k}=s_{i_1}\cdots s_{i_k}$, the inclusion
$\mathscr C_{w_{\leq k}}\subseteq\mathscr C_w$ extends its
PBW parameters by zeros in positions $k+1,\ldots,r$.
Moreover, if $X,Y\in\mathscr C_w$ are simple and $X\circ Y$ is
simple, then
\begin{equation}\label{eq:pbw-simple-product}
 \ba^{\bw}(X\circ Y)=\ba^{\bw}(X)+\ba^{\bw}(Y).
\end{equation}
The last identity follows from the leading term of the dual PBW
expansion. These properties follow from the PBW and crystal
formalism in \cite[Section~4]{kimura2012quantum}, together with
the cuspidal description of $\mathscr C_w$ in
\cite[Section~2]{kkop2018strata}.

\section{Lusztig parameters of determinantial modules}\label{sec:parameters}

\subsection{Leftmost subexpressions}

Fix $v\leq w$ and a reduced expression
$\bw=(i_1,\ldots,i_r)$ of $w$. The \emph{leftmost reduced
subexpression} of $v$ in $\bw$ is obtained as follows. Starting with
$z_0=v$, define recursively
\[
 z_e=
 \begin{cases}
   s_{i_e}z_{e-1},&s_{i_e}z_{e-1}<z_{e-1},\\
   z_{e-1},&s_{i_e}z_{e-1}>z_{e-1}.
 \end{cases}
\]
Select position $e$ in the first case. The selected letters form a
reduced expression $\bv$ of $v$, and $z_r=1$.
Let $J_{\bw,v}$ be the set of unselected positions.
For $0\leq k\leq r$, put
\[
 \bw_{\leq k}=(i_1,\ldots,i_k),\qquad
 w_{\leq k}=s_{i_1}\cdots s_{i_k},\qquad
 v_{\leq k}=\prod_{\substack{1\leq d\leq k\\d\notin J_{\bw,v}}}s_{i_d}.
\]
All products indexed by positions are taken in increasing order;
empty products are the identity. In particular, $w_{\leq0}=v_{\leq0}=1$.
The selected positions in $[1,k]$ form the leftmost reduced
subexpression of $v_{\leq k}$ in $\bw_{\leq k}$.

For $1\leq e\leq k\leq r$, define the selected suffix
\begin{equation}\label{eq:selected-suffix}
 \sigma_{e,k}
 :=\prod_{\substack{e<d\leq k\\d\notin J_{\bw,v}}}s_{i_d}.
\end{equation}
For $\Lambda\in P_+$, define $\bdelta_{\bw,v}(\Lambda)\in\ZZ^r$ by
\begin{equation}\label{eq:delta-definition}
 \bigl(\bdelta_{\bw,v}(\Lambda)\bigr)_e
 =
 \begin{cases}
  \bigl\langle\sigma_{e,r}^{-1}(\alpha_{i_e}^\vee),\Lambda\bigr\rangle,
      &e\in J_{\bw,v},\\
  0,&e\notin J_{\bw,v}.
 \end{cases}
\end{equation}
We abbreviate $\bdelta_{\bw,v}(\varpi_j)$ to
$\bdelta_{\bw,v}^{\,j}$.
If $e\in J_{\bw,v}$, the greedy construction gives
$s_{i_e}\sigma_{e,r}>\sigma_{e,r}$; hence the coroot in
\eqref{eq:delta-definition} is positive. Thus
$\bdelta_{\bw,v}(\Lambda)\in\ZZ_{\geq0}^r$.

\subsection{Simultaneous descent}

\begin{lemma}\label{lem:determinantial-simultaneous-descent}
Let $\Lambda\in P_+$ and $y\leq x$. If
$s_i x<x$ and $s_i y<y$, then
\[
 \varepsilon_i^*\bigl(M(x\Lambda,y\Lambda)\bigr)=0,
 \qquad
 T_iM(x\Lambda,y\Lambda)
 \simeq M(s_i x\Lambda,s_i y\Lambda).
\]
\end{lemma}

\begin{proof}
Set $\lambda=x\Lambda$, $\mu=y\Lambda$, and $X=M(\lambda,\mu)$.
The descent assumptions imply
\[
 a=-\langle h_i,\lambda\rangle\geq0,
 \qquad b=-\langle h_i,\mu\rangle\geq0.
\]
If $\varepsilon_i^*(X)>0$, the restriction-support property~\eqref{eq:determinantial-support} implies
that $\mu-\alpha_i$ is a weight of
the integrable highest-weight module $V(\Lambda)$.
By Weyl invariance of its weight set, $\Lambda-y^{-1}\alpha_i$
would then also be a weight. But $y^{-1}\alpha_i$ is a negative
root, contradicting the highest-weight property of $\Lambda$.
Consequently, $\varepsilon_i^*(X)=0$.

Similarly, $\widetilde E_i^kX\neq0$ implies that
$\lambda+k\alpha_i$ is a weight of $V(\Lambda)$.
The $i$-string through the extremal weight $\lambda$ ends at
$s_i\lambda=\lambda+a\alpha_i$, so $\varepsilon_i(X)\leq a$.
It follows that
\[
 \varphi_i(X)=\varepsilon_i(X)+b-a\leq b.
\]
We may therefore apply~\eqref{eq:saito-stabilized} with $s=b$ to obtain
\[
 T_iX=\widetilde E_i^{\,a}(\widetilde F_i^*)^bM(\lambda,\mu).
\]

The lifting property of Bruhat order gives $s_i y\leq s_i x$.
Moreover, $s_i y\leq y\leq x$, so all determinantial modules in
the following applications of~\eqref{eq:star-extremal-string}
and~\eqref{eq:unstarred-extremal-string} are defined:
\[
 (\widetilde E_i^*)^bM(\lambda,s_i\mu)
      \simeq M(\lambda,\mu),\qquad
 \widetilde E_i^{\,a}M(\lambda,s_i\mu)
      \simeq M(s_i\lambda,s_i\mu).
\]
Inverting the first identity along its crystal string yields
$(\widetilde F_i^*)^bM(\lambda,\mu)\simeq M(\lambda,s_i\mu)$.
Substitution proves the assertion.
\end{proof}

\subsection{The parameter formula}

\begin{theorem}\label{thm:lusztig-determinantial}
Let $\Lambda\in P_+$, $v\leq w$, and let $\bw$ be a reduced
expression of $w$. Then
\[
 \ba^{\bw}\bigl(M(w\Lambda,v\Lambda)\bigr)
       =\bdelta_{\bw,v}(\Lambda).
\]
In particular,
$\ba^{\bw}\bigl(M(w\varpi_j,v\varpi_j)\bigr)
=\bdelta_{\bw,v}^{\,j}$ for every $j\in I$.
\end{theorem}

\begin{proof}
We argue by induction on $r=\ell(w)$. For $r=0$, we have $v=w=1$,
and the assertion is immediate. Suppose $r>0$, and set
\[
 i=i_1,\qquad w'=s_iw,\qquad \bw'=(i_2,\ldots,i_r),\qquad
 X=M(w\Lambda,v\Lambda).
\]
Write $\ba^{\bw}(X)=(a_1,\ldots,a_r)$; in particular,
$a_1=\varepsilon_i^*(X)$ by~\eqref{eq:lusztig-saito-recursion}.

If $1\notin J_{\bw,v}$, then $s_iv<v$. Removing the first selected
letter leaves the leftmost reduced subexpression of $v'=s_iv$ in
$\bw'$. Lemma~\ref{lem:determinantial-simultaneous-descent} gives
\[
 a_1=0,\qquad T_iX\simeq M(w'\Lambda,v'\Lambda).
\]
The induction hypothesis and \eqref{eq:lusztig-saito-recursion}
therefore determine the remaining coordinates.

If $1\in J_{\bw,v}$, then $s_iv>v$, and the selected subexpression
lies entirely in $\bw'$. Thus $v\leq w'$. Since $s_iw<w$, the
lifting property also gives $s_iv\leq w$.
Put $n=\langle h_i,v\Lambda\rangle\geq0$.
The starred extremal-string identity~\eqref{eq:star-extremal-string} gives
\[
 a_1=\varepsilon_i^*(X)=n,\qquad
 (\widetilde E_i^*)^nX\simeq M(w\Lambda,s_iv\Lambda).
\]
Applying Lemma~\ref{lem:determinantial-simultaneous-descent} to
the pair $(w,s_iv)$ yields
\[
 T_i(\widetilde E_i^*)^nX\simeq M(w'\Lambda,v\Lambda).
\]
Here $\sigma_{1,r}=v$, so
\[
 a_1=\langle h_i,v\Lambda\rangle
 =\bigl\langle\sigma_{1,r}^{-1}(\alpha_i^\vee),\Lambda\bigr\rangle.
\]
Again, the remaining coordinates follow from the induction
hypothesis and \eqref{eq:lusztig-saito-recursion}.

In either case, deleting the first position leaves the selected
suffix after every later position unchanged. The resulting
coordinates are therefore precisely those in
\eqref{eq:delta-definition}.
\end{proof}

\begin{remark}[Compatibility with weights]\label{rem:parameter-weight}
The formula satisfies the weight identity
\[
 \sum_{e=1}^r
   \bigl(\bdelta_{\bw,v}(\Lambda)\bigr)_e\rho_e
 =v\Lambda-w\Lambda.
\]
This can also be checked without crystal operators.
Set $\sigma_{0,r}=v$ and
$\mu_e=w_{\leq e}\sigma_{e,r}\Lambda$ for $0\leq e\leq r$.
At a selected position, $\mu_{e-1}=\mu_e$; at an unselected
position, their difference is
$\langle\alpha_{i_e}^\vee,\sigma_{e,r}\Lambda\rangle\rho_e$.
Summing these differences gives the identity.
\end{remark}

\begin{example}[A pair outside the length-additive case]\label{ex:A2-parameters}
In type $A_2$, take
\[
 \bw=(1,2),\qquad w=s_1s_2,\qquad v=s_2.
\]
Then $v\leq w$, but
$\ell(v^{-1}w)=3\neq\ell(w)-\ell(v)=1$.
The selected position is $2$, so $J_{\bw,v}=\{1\}$ and
$\sigma_{1,2}=s_2$.
For $\Lambda=a\varpi_1+b\varpi_2$, with $a,b\geq0$, we obtain
\[
 s_2(\alpha_1^\vee)=\alpha_1^\vee+\alpha_2^\vee,
 \qquad
 \bdelta_{\bw,v}(\Lambda)=(a+b,0).
\]
Thus Theorem~\ref{thm:lusztig-determinantial} gives
\[
 M(s_1s_2\Lambda,s_2\Lambda)
 \simeq L(1)^{\odot(a+b)},
\]
where $L(1)$ is the self-dual simple module of weight $-\alpha_1$.
Here the selected suffix contributes to the parameter even though
its own coordinate is zero.
\end{example}

For later use, let $\iota_k:\ZZ^k\hookrightarrow\ZZ^r$ denote
extension by zeros in the last $r-k$ positions, and set
\begin{equation}\label{eq:prefix-delta}
 \bdelta_k(\Lambda)
 :=\iota_k\bigl(\bdelta_{\bw_{\leq k},v_{\leq k}}(\Lambda)\bigr),
 \qquad
 \bdelta_{k,j}:=\bdelta_k(\varpi_j).
\end{equation}
Explicitly,
\begin{equation}\label{eq:prefix-delta-coordinates}
 (\bdelta_{k,j})_e
 =\begin{cases}
 \bigl\langle\sigma_{e,k}^{-1}(\alpha_{i_e}^\vee),\varpi_j\bigr\rangle,
       &e\in J_{\bw,v},\ e\leq k,\\
 0,&\text{otherwise}.
 \end{cases}
\end{equation}
Compatibility of PBW parametrizations with the inclusion
$\mathscr C_{w_{\leq k}}\subseteq\mathscr C_w$ gives the following consequence.

\begin{corollary}\label{cor:prefix-parameters}
For $1\leq k\leq r$ and $\Lambda\in P_+$,
\[
 \ba^{\bw}\bigl(M(w_{\leq k}\Lambda,v_{\leq k}\Lambda)\bigr)
 =\bdelta_k(\Lambda).
\]
In particular, the parameter for $\Lambda=\varpi_j$ is
$\bdelta_{k,j}$.
\end{corollary}

\begin{proof}
Apply Theorem~\ref{thm:lusztig-determinantial} to
$(\bw_{\leq k},v_{\leq k})$ and extend its parameter vector by zeros.
\end{proof}

\section{Demazure weaves in finite simply-laced type}\label{sec:weaves}

Throughout this section, assume that $\fg$ is of finite type
$A$, $D$, or $E$, and let $w_0$ be the longest element
of $W$. Keep $v\leq w$ and $\bw=(i_1,\ldots,i_r)$ fixed.

\subsection{Demazure products and Lusztig cycles}

For $x\in W$ and $i\in I$, define the Demazure product by
\[
 x*s_i=
 \begin{cases}
 xs_i,&\ell(xs_i)=\ell(x)+1,\\
 x,&\ell(xs_i)=\ell(x)-1.
 \end{cases}
\]
For a word $\beta=(b_1,\ldots,b_t)$, put
$\delta(\beta)=s_{b_1}*\cdots*s_{b_t}$.
A \emph{Demazure weave} records braid moves and contractions
$ii\to i$ that reduce $\beta$ to a reduced word for $\delta(\beta)$.
We draw the word at the top and perform these moves downwards.
In simply-laced type, commutations and three-term braid moves
give four- and six-valent vertices, respectively; contractions
give trivalent vertices.

The \emph{right inductive weave}
$\overrightarrow{\mathfrak W}(\beta)$ processes the letters from
left to right. At each step, the new letter is inserted on the
right of a reduced representative of the preceding Demazure
product. An ascent extends this representative. At a descent,
braid moves bring the repeated letter to its right end, where
the contraction $ii\to i$ creates one trivalent vertex.
We fix such braid moves throughout; see
\cite[Section~4]{casals2025cluster} and
\cite[Section~5.1]{bi2026monoidal}.

The distinguished \emph{Lusztig cycle} $\gamma_u$ of a trivalent
vertex $u$ assigns nonnegative integer weights to the edges.
It vanishes above $u$, has weight $1$ on the outgoing edge of $u$,
and propagates downwards by the tropical Lusztig rules.
A commutation swaps the two weights, and a three-term braid move
acts by
\[
 (a,b,c)\longmapsto(b+c-p,p,a+b-p),
 \qquad p=\min\{a,c\},
\]
with the weights read from left to right.
At every subsequent trivalent vertex, the outgoing weight is the
minimum of the two incoming weights. The initialization at $u$
is the only exception to this last rule.
These rules are the cycle conventions of
\cite[Section~5.2]{bi2026monoidal}.

\subsection{Weave vectors and cluster-variable modules}

Choose a reduced word $\mathbf h$ for $w_0v^{-1}$,
and put
\[
 \beta=\mathbf h\bw,\qquad
 \mathfrak W_{\bw,v}=\overrightarrow{\mathfrak W}(\beta),\qquad
 V_{\bw,v}=(\mathfrak W_{\bw,v})_3.
\]
Here the arrow denotes the right inductive Demazure weave, and the
subscript $3$ denotes its set of trivalent vertices.
The Demazure product of $\beta$ is $w_0$.
The letters of $\mathbf h$ introduce no trivalent vertices, and the
vertices created in the $\bw$-block correspond exactly to
$J_{\bw,v}$. Indeed, for $0\leq k\leq r$,
\begin{equation}\label{eq:prefix-demazure}
 \delta\bigl(\mathbf h\bw_{\leq k}\bigr)
       =w_0v^{-1}v_{\leq k}.
\end{equation}
To see this, write $t_k=v_{\leq k}^{-1}v$. The greedy rule selects $k$
exactly when $s_{i_k}t_{k-1}<t_{k-1}$. Since
$w_0v^{-1}v_{\leq k-1}=w_0t_{k-1}^{-1}$,
length reversal by the longest element identifies this condition
with an ascent after right multiplication by $s_{i_k}$.
This proves \eqref{eq:prefix-demazure} by the Demazure recursion.
Unselected positions are therefore exactly the trivalent steps.
We write
\[
 V_{\bw,v}\xrightarrow{\sim}J_{\bw,v},\qquad u\longmapsto q_u,
\]
and order the vertices by $u'<u$ if $q_{u'}<q_u$.

We first specify the modules whose Lusztig parameters we wish
to compute. By
\cite[Theorem~10.1 and Section~10.1]{casals2025cluster},
the seed associated with
$\overrightarrow{\mathfrak W}(\mathbf h\bw)$
is identified with M\'enard's seed $\mathbf s(v,\bw)$.
Here we use the right-inductive, leftmost-subexpression
convention, obtained from the convention in that reference
by reversing words and replacing $(v,w)$ with
$(v^{-1},w^{-1})$.

We use the quantum lift of M\'enard's seed obtained by applying
its prescribed mutation sequence to the initial monoidal seed
of $\mathscr C_w$, then freezing and retaining the prescribed
vertices. The weave comparison identifies the classical
mutation and deletion data; the monoidal categorification
theorem of~\cite{kkko2018monoidal} realizes these mutations
by modules.
For $u\in V_{\bw,v}$, let $M_u$ be the resulting self-dual real
simple module at $u$. Since these modules belong to a common
monoidal seed, they are pairwise strongly commuting.

Let $\gamma_u$ be the Lusztig cycle attached to $u$, and let
$e_{u,\mathrm{nw}}$ be its northwest incoming edge.
Write $E_e$ for the $e$th standard basis vector of $\ZZ^r$.
We define the weave vectors recursively by
\begin{equation}\label{eq:etauw}
 \eta_u^{\bw}
 :=
 E_{q_u}
 +\sum_{u'<u}
    \gamma_{u'}(e_{u,\mathrm{nw}})\eta_{u'}^{\bw}.
\end{equation}
The order is the one induced by the positions $q_u$ in
$\bw$. Thus the right-hand side involves only vectors
already defined, and the recursion determines
$\eta_u^{\bw}$ uniquely.
Our aim is to show that these combinatorial vectors are
precisely the Lusztig parameters of the modules $M_u$.

We recall the parameter rule used in this comparison.
For vectors in $\ZZ^r$, the right lexicographic order compares
the entries at the largest index where they differ.
We write $\max_{\mathrm r}$ for the maximum in this order.

\begin{lemma}[Mutation of parameters]\label{lem:parameter-mutation}
Let $(N_s)_s$ be the modules of a monoidal seed in
$\mathscr C_w$, with exchange matrix $B=(b_{st})$, and put
$\ba_s=\ba^{\bw}(N_s)$.
If mutation at $t$ replaces $N_t$ by $N_t'$, then
\begin{equation}\label{eq:parameter-mutation}
 \ba^{\bw}(N_t')
 =\max_{\mathrm r}\left\{
       \sum_s[b_{st}]_+\ba_s,
       \sum_s[-b_{st}]_+\ba_s
   \right\}-\ba_t,
 \qquad [b]_+:=\max\{b,0\}.
\end{equation}
All other parameter vectors remain unchanged.
\end{lemma}

\begin{proof}
Fix an order of the factors in the quantum exchange relation
\[
 q^{c_0}[N_t][N_t']
 =q^{c_+}\prod_s[N_s]^{[b_{st}]_+}
  +q^{c_-}\prod_s[N_s]^{[-b_{st}]_+}.
\]
The leading dual PBW parameter of a product is the sum of
the parameters of its factors, and its leading coefficient
is a power of $q^{1/2}$.
The two leading contributions on the right cannot cancel.
Consequently, the leading parameter of that side is the
right lexicographic maximum of the parameters of its two
monomials. Comparing with the left-hand side gives
\eqref{eq:parameter-mutation}.
This is the argument of
\cite[Lemma~7.2]{bi2026monoidal}, applied in the ordinary
quantum unipotent coordinate ring.
\end{proof}

The maximum in~\eqref{eq:parameter-mutation} is taken in a
total order; it is not a coordinatewise maximum.
The following two compatibility results explain why the same
rule can be used to compare modules and weave vectors.

\begin{lemma}[Transport of recursive vectors]\label{lem:weave-transport}
Start with the recursively defined vectors of a right
inductive weave for a word $\boldsymbol\alpha$.
Move its first right insertion to a final left insertion,
and apply the induced mutations to the vectors using
\eqref{eq:parameter-mutation}.
The resulting vectors satisfy
\[
 \zeta_z=E_{p_z}
       +\sum_{y<z}\gamma_y^{\mathrm{new}}(e_{z,\mathrm{nw}})\zeta_y
\]
at every trivalent vertex except possibly the one created
by the final left insertion.
Here $p_z$ is the position in $\boldsymbol\alpha$ of the letter
creating $z$, and the cycles and vertex order are those of
the resulting weave.
\end{lemma}

\begin{proof}
When the transport contains a mutation, this is
\cite[Corollary~6.10]{bi2026monoidal}.
The intervening weave equivalences preserve the relevant
cycle weights by \cite[Lemma~6.7]{bi2026monoidal}.
The same equivalences give the assertion when no mutation
occurs.
\end{proof}

\begin{lemma}[Extension from a truncated weave]
\label{lem:truncated-weave-mutations}
Let $\mathfrak U$ be a truncated double inductive weave with
vertex set $S$, and let $\mathfrak W$ be obtained by restoring
the remaining insertions, with vertex set $K\supseteq S$.
For every vertex $t$ mutable in $\mathfrak U$,
\begin{equation}\label{eq:truncated-weave-exchange}
 b^{\mathfrak W}_{zt}
 =\begin{cases}
    b^{\mathfrak U}_{zt},&z\in S,\\
    0,&z\in K\setminus S.
   \end{cases}
\end{equation}
Every mutation sequence in the truncated weave extends to
the full weave with the same exchange monomials.
The variables at the additional vertices remain unchanged.
\end{lemma}

\begin{proof}
This is \cite[Lemma~7.6]{bi2026monoidal}.
The vanishing in~\eqref{eq:truncated-weave-exchange} means
that the additional vertices do not occur in the exchange
monomials. This property persists along the mutation sequence,
so the vector updates agree as well.
\end{proof}

\begin{proposition}\label{prop:weave-parameters}
For every $u\in V_{\bw,v}$, one has
\[
 \ba^{\bw}(M_u)=\eta_u^{\bw}.
\]
\end{proposition}

\begin{proof}
Let $\bv=(j_1,\ldots,j_m)$ be the reduced expression of
$v$ selected by the leftmost reduced subexpression of
$\bw$, and define $i^*\in I$ by
$s_{i^*}w_0=w_0s_i$.
Since $\mathbf h$ is reduced for $w_0v^{-1}$, the word
\[
 \boldsymbol\Delta
 :=
 (j_1^*,\ldots,j_m^*,\mathbf h)
\]
is reduced for $w_0$.
Indeed, its product is
$(w_0vw_0)(w_0v^{-1})=w_0$, and its length is
$\ell(w_0)$.

We begin with
$\overrightarrow{\mathfrak W}(\boldsymbol\Delta\bw)$.
Its seed is identified with the usual initial seed of
$\mathscr C_w$; see
\cite[Section~7, equation~(7.2)]{bi2026monoidal}.
The module at the vertex corresponding to position $k$
of $\bw$ is
\[
 L_k=M(w_{\leq k}\varpi_{i_k},\varpi_{i_k}).
\]
By Theorem~\ref{thm:lusztig-determinantial},
\begin{equation}\label{eq:initial-weave-parameters}
 \ba^{\bw}(L_k)
 =
 \sum_{\substack{1\leq a\leq k\\i_a=i_k}}E_a.
\end{equation}
These are also the initial weave vectors.
Indeed, if $k^-$ is the preceding occurrence of $i_k$
in $\bw$, then
\cite[Lemma~6.12]{bi2026monoidal} gives
\[
 \eta_k^{(0)}=E_k+\eta_{k^-}^{(0)},
\]
with the second term omitted when there is no preceding
occurrence. Iterating this identity gives exactly
\eqref{eq:initial-weave-parameters}.

The combinatorial vectors of the full weave naturally have
$\ell+r$ coordinates, where $\ell=\ell(w_0)$.
The first $\ell$ coordinates vanish because the reduced
$\boldsymbol\Delta$-block creates no trivalent vertices.
Use the zero extension
\[
 \iota:\ZZ^r\hookrightarrow\ZZ^{\ell+r},\qquad
 \iota(E_k)=E_{\ell+k}
\]
to compare them with the categorical parameters in
\eqref{eq:initial-weave-parameters}.
This embedding preserves addition and right lexicographic order:
\[
 \max_{\mathrm r}\{\iota(\mathbf a),\iota(\mathbf b)\}
 =\iota\bigl(\max_{\mathrm r}\{\mathbf a,\mathbf b\}\bigr).
\]
Consequently, vector mutation preserves $\iota(\ZZ^r)$ and
commutes with $\iota$.
We suppress this zero extension below and record all vectors
in $\ZZ^r$.
The enlarged lattice labels combinatorial weave vectors;
the categorical parameters always refer to the reduced word $\bw$.

For $0\leq s\leq m$, let $\mathfrak W^{(s)}$ be the
weave associated with the double string
\[
 \bigl(
 j_{s+1}^*\mathrm R,\ldots,j_m^*\mathrm R,
 \mathbf h\,\mathrm R,\bw\,\mathrm R,
 j_s^*\mathrm L,\ldots,j_1^*\mathrm L
 \bigr).
\]
A block marked $\mathrm R$ means that each of its letters
is inserted on the right.
Thus
\[
 \mathfrak W^{(0)}
 =
 \overrightarrow{\mathfrak W}(\boldsymbol\Delta\bw),
 \qquad
 \mathfrak W^{(m)}
 =
 \mathfrak W\bigl(
 \mathbf h\,\mathrm R,\bw\,\mathrm R,
 j_m^*\mathrm L,\ldots,j_1^*\mathrm L
 \bigr).
\]
The passage from $\mathfrak W^{(s-1)}$ to
$\mathfrak W^{(s)}$ moves $j_s^*$ from the initial
right-insertion block to a left insertion immediately
before the previously transported letters.
Apply the corresponding mutations and relabelings to
the monoidal seed, and denote the resulting parameter
at a vertex $z$ by $\mathbf d_z^{(s)}$.

We show inductively that, after omitting the final left
insertions, the parameters in the $\bw$-block form the
recursively defined vector family of the remaining right
inductive weave, embedded in the fixed coordinate lattice.
The initial case follows from
\eqref{eq:initial-weave-parameters}.
For the $s$th step, temporarily omit the final
insertions
$j_{s-1}^*\mathrm L,\ldots,j_1^*\mathrm L$.
The remaining weave is right inductive, so Lemma~\ref{lem:weave-transport} applies to the transport of $j_s^*$.
The module parameters undergo precisely the required
vector mutations by~\eqref{eq:parameter-mutation}.
They therefore satisfy the recursion at every vertex
except possibly the new final vertex attached to
$j_s^*\mathrm L$.

Restoring the omitted insertions does not alter this
calculation: by Lemma~\ref{lem:truncated-weave-mutations} and
\eqref{eq:truncated-weave-exchange},
the mutations have the same exchange monomials in the
full weave.
Moreover, the exceptional vertex and all previously
transported left insertions occur after the entire
$\bw$-block.
They do not enter the recursion at a vertex in that
block.
The remaining initial block
\[
 (j_{s+1}^*,\ldots,j_m^*,\mathbf h)
\]
is a reduced suffix of $\boldsymbol\Delta$, so it
contains no trivalent vertices.
Hence the recursion at each vertex of the $\bw$-block
involves only earlier vertices of the same block.
This completes the induction.

After the $m$ transports, delete the final left
insertions.
The remaining weave is
$\overrightarrow{\mathfrak W}(\mathbf h\bw)$.
By the comparison with M\'enard's construction in
\cite[Section~10.1]{casals2025cluster},
the surviving modules are precisely the $M_u$.
Their parameters are unchanged by deletion.
The cycles on this retained initial part of the weave
are also unchanged, so the recursion just proved becomes
\[
 \ba^{\bw}(M_u)
 =
 E_{q_u}
 +\sum_{u'<u}
   \gamma_{u'}(e_{u,\mathrm{nw}})
   \ba^{\bw}(M_{u'}).
\]
This is the defining recursion~\eqref{eq:etauw}.
Induction on the ordered set $V_{\bw,v}$ therefore gives
$\ba^{\bw}(M_u)=\eta_u^{\bw}$ for every $u$.
\end{proof}

\begin{corollary}\label{cor:weave-support}
For every $u\in V_{\bw,v}$,
\[
 \supp\ba^{\bw}(M_u)\subseteq J_{\bw,v}\cap[1,q_u],
 \qquad \bigl(\ba^{\bw}(M_u)\bigr)_{q_u}=1.
\]
\end{corollary}

\begin{proof}
By Proposition~\ref{prop:weave-parameters}, it suffices to prove
the assertions for $\eta_u^{\bw}$. Proceed by induction in the
order of the vertices. Every vector $\eta_{u'}^{\bw}$ appearing
in \eqref{eq:etauw} is supported in
$J_{\bw,v}\cap[1,q_{u'}]$, where $q_{u'}<q_u$.
The added term $E_{q_u}$ therefore gives both the asserted support
and the coefficient $1$ at position $q_u$.
\end{proof}

\subsection{Cycle multiplicities and determinantial factorizations}\label{subsec:factorization}

For $0\leq k\leq r$, let
\[
 \beta_k=\mathbf h(i_1,\ldots,i_k),\qquad
 d_k=\delta(\beta_k).
\]
Choose a horizontal slice after the first $k$ letters of the $\bw$-block have been
processed. Reading its edges from left to right gives a reduced
expression
$d_k=s_{j_1}\cdots s_{j_{\ell_k}}$, where $\ell_k=\ell(d_k)$.
Denote these edges by
$\mathsf e_1^{(k)},\ldots,\mathsf e_{\ell_k}^{(k)}$ and put
\begin{equation}\label{eq:inversion-coroot-c}
 \chi_a^{(k)}
 :=s_{j_{\ell_k}}\cdots s_{j_{a+1}}(\alpha_{j_a}^\vee),
 \qquad 1\leq a\leq\ell_k.
\end{equation}
For $u\in V_{\bw,v}$ with $q_u\leq k$, define
\begin{equation}\label{eq:gamma-vc}
 \gamma_{u,k}
 :=\sum_{a=1}^{\ell_k}
       \gamma_u(\mathsf e_a^{(k)})\chi_a^{(k)}\in Q^\vee,
 \qquad
 m_{u,k}^{\,j}:=\langle\gamma_{u,k},\varpi_j\rangle.
\end{equation}
More generally, put
$m_{u,k}(\Lambda):=\langle\gamma_{u,k},\Lambda\rangle$
for $\Lambda\in P_+$.
The inversion coroots are positive and Lusztig-cycle weights
are nonnegative integers, so all these multiplicities are
nonnegative integers.
The coroot $\gamma_{u,k}$ is independent of braid moves changing
the chosen reduced expression; see
\cite[Lemma~6.15]{bi2026monoidal}.

\begin{theorem}\label{thm:determinantial-factorization}
For $1\leq k\leq r$ and $\Lambda\in P_+$, there is an isomorphism
\begin{equation}\label{eq:dominant-determinantial-factors}
 M(w_{\leq k}\Lambda,v_{\leq k}\Lambda)
 \simeq
 \bigodot_{\substack{u\in V_{\bw,v}\\q_u\leq k}}
       M_u^{\odot m_{u,k}(\Lambda)}.
\end{equation}
In particular, for every $j\in I$,
\begin{equation}\label{eq:determinantial-factors}
 M(w_{\leq k}\varpi_j,v_{\leq k}\varpi_j)
 \simeq
 \bigodot_{\substack{u\in V_{\bw,v}\\q_u\leq k}}
       M_u^{\odot m_{u,k}^{\,j}}.
\end{equation}
\end{theorem}

\begin{proof}
Set $L=\ell(w_0v^{-1})$, the length of $\mathbf h$.
The unselected positions of $\beta=\mathbf h\bw$ are
$L+J_{\bw,v}$ by~\eqref{eq:prefix-demazure}, and its weave vectors
vanish in the first $L$ coordinates.
Apply the coroot identity of \cite[Theorem~6.17]{bi2026monoidal}
at positions $L+e$ and $L+k$.
After removing these initial zero coordinates, the selected
suffix is precisely $\sigma_{e,k}$. Thus, for every
$e\in J_{\bw,v}\cap[1,k]$,
\begin{equation}\label{eq:cycle-coroot-identity}
 \sum_{\substack{u\in V_{\bw,v}\\e\leq q_u\leq k}}
     (\eta_u^{\bw})_e\,\gamma_{u,k}
 =\sigma_{e,k}^{-1}(\alpha_{i_e}^\vee).
\end{equation}
Pairing with $\Lambda$ and using the definition of
$\bdelta_k(\Lambda)$ and Corollary~\ref{cor:weave-support} gives
\[
 \bdelta_k(\Lambda)
 =\sum_{\substack{u\in V_{\bw,v}\\q_u\leq k}}
       m_{u,k}(\Lambda)\eta_u^{\bw}.
\]
The coordinates outside $J_{\bw,v}\cap[1,k]$ vanish on both sides.
For fundamental weights this becomes
\begin{equation}\label{eq:prefix-vector-factorization}
 \bdelta_{k,j}
 =\sum_{\substack{u\in V_{\bw,v}\\q_u\leq k}}
       m_{u,k}^{\,j}\eta_u^{\bw}.
\end{equation}

The product on the right-hand side of
\eqref{eq:dominant-determinantial-factors} is self-dual and simple,
since its factors are real and pairwise strongly commuting.
By \eqref{eq:pbw-simple-product} and
Proposition~\ref{prop:weave-parameters}, its Lusztig parameter is
$\bdelta_k(\Lambda)$.
Corollary~\ref{cor:prefix-parameters} gives the same parameter
for the determinantial module on the left.
Injectivity of the parametrization and self-dual normalization
prove the isomorphism.
\end{proof}

\begin{corollary}[Quantum-minor factorization]\label{cor:quantum-minor-factorization}
Let $A_u=[M_u]$ under quiver Hecke categorification.
For $1\leq k\leq r$ and $\Lambda\in P_+$,
\[
 D(w_{\leq k}\Lambda,v_{\leq k}\Lambda)
 =\bigodot_{q_u\leq k}
       A_u^{\odot\langle\gamma_{u,k},\Lambda\rangle}.
\]
Here $\odot$ uses the same power of $q$ as the self-dual
convolution in Theorem~\ref{thm:determinantial-factorization}.
\end{corollary}

\begin{proof}
Take Grothendieck classes in
\eqref{eq:dominant-determinantial-factors}.
\end{proof}

\begin{remark}\label{rem:prefix-indexing}
The upper bound $q_u\leq k$ in
\eqref{eq:determinantial-factors} is inclusive: a vertex created
at the $k$th step can contribute to the minor at that step.
Also, the relevant parameter is the prefix vector
$\bdelta_{k,j}$, which agrees with $\bdelta_{\bw,v}^{\,j}$ when
$k=r$ but need not agree with it for $k<r$.
\end{remark}
\subsection{Recognizing seed monomials and recovering multiplicities}

The triangular recursion~\eqref{eq:etauw} also recovers the
multiplicities of a seed monomial directly from its parameters.
List the trivalent vertices as $u_1<\cdots<u_N$, and define
the coordinate projection
\[
 \pi_J:\ZZ^r\longrightarrow\ZZ^N,\qquad
 \pi_J(\ba)=(a_{q_{u_1}},\ldots,a_{q_{u_N}})^{\mathsf T}.
\]
Let $C=(c_{ab})$ be the strictly upper triangular matrix
\[
 c_{ab}=
 \begin{cases}
 \gamma_{u_a}(e_{u_b,\mathrm{nw}}),&a<b,\\
 0,&a\geq b,
 \end{cases}
\]
and let the $b$th column of $H$ be $\pi_J(\eta_{u_b}^{\bw})$.
The recursion~\eqref{eq:etauw} gives
\begin{equation}\label{eq:triangular-inverse}
 H=I_N+HC,\qquad H^{-1}=I_N-C.
\end{equation}
Thus $H$ is an upper unitriangular integer matrix.
Here a \emph{monomial in the seed} means a normalized convolution
of the $M_u$ with nonnegative integer exponents, including at
frozen vertices.

\begin{proposition}\label{prop:seed-monomial}
Let $X$ be a self-dual simple module in $\mathscr C_w$, and put
$\ba=\ba^{\bw}(X)$. Then $X$ is a monomial in $\mathbf s(v,\bw)$
if and only if
\[
 \supp\ba\subseteq J_{\bw,v},
 \qquad
 \mathbf b:=(I_N-C)\pi_J(\ba)\in\ZZ_{\geq0}^N.
\]
When these conditions hold, its unique factorization in this seed is
\[
 X\simeq\bigodot_{a=1}^N M_{u_a}^{\odot b_a},
 \qquad
 b_a=a_{q_{u_a}}-
       \sum_{b>a}\gamma_{u_a}(e_{u_b,\mathrm{nw}})\,a_{q_{u_b}}.
\]
\end{proposition}

\begin{proof}
Suppose first that $X\simeq\bigodot_a M_{u_a}^{\odot b_a}$ with
$\mathbf b\in\ZZ_{\geq0}^N$. Additivity and
Proposition~\ref{prop:weave-parameters} give
$\ba=\sum_a b_a\eta_{u_a}^{\bw}$.
Corollary~\ref{cor:weave-support} gives the support condition,
and projection yields $\pi_J(\ba)=H\mathbf b$.
Equation~\eqref{eq:triangular-inverse} now determines $\mathbf b$.

Conversely, suppose that the two stated conditions hold.
The normalized convolution
$Y=\bigodot_a M_{u_a}^{\odot b_a}$ is self-dual and simple.
Its parameter is supported in $J_{\bw,v}$, and
\[
 \pi_J\bigl(\ba^{\bw}(Y)\bigr)
      =H\mathbf b=\pi_J(\ba).
\]
The support condition therefore gives
$\ba^{\bw}(Y)=\ba^{\bw}(X)$ in the full lattice $\ZZ^r$.
Injectivity implies $X\simeq Y$.
The displayed coordinate formula and uniqueness follow from
$\mathbf b=(I_N-C)\pi_J(\ba)$.
\end{proof}

\begin{corollary}\label{cor:multiplicity-from-parameters}
For $1\leq k\leq r$, $j\in I$, and $u\in V_{\bw,v}$ with
$q_u\leq k$, the multiplicity in
Theorem~\ref{thm:determinantial-factorization} is
\begin{equation}\label{eq:multiplicity-from-parameters}
 m_{u,k}^{\,j}
 =(\bdelta_{k,j})_{q_u}
  -\sum_{\substack{x>u\\q_x\leq k}}
       \gamma_u(e_{x,\mathrm{nw}})(\bdelta_{k,j})_{q_x}.
\end{equation}
The same formula holds for $m_{u,k}(\Lambda)$ after replacing
$\bdelta_{k,j}$ by $\bdelta_k(\Lambda)$.
\end{corollary}

\begin{proof}
Apply Proposition~\ref{prop:seed-monomial} to the factorization
in Theorem~\ref{thm:determinantial-factorization}.
The parameter of the determinantial module is
$\bdelta_k(\Lambda)$ by Corollary~\ref{cor:prefix-parameters}.
Its coordinates beyond $k$ vanish, so the terms with $q_x>k$
drop out of the multiplicity formula.
Taking $\Lambda=\varpi_j$ gives
\eqref{eq:multiplicity-from-parameters}.
\end{proof}

\begin{remark}[Computing the multiplicities]\label{rem:multiplicity-algorithm}
The formula gives a calculation in two stages.
First compute $\bdelta_k(\Lambda)$ by scanning the positions
$e=k,k-1,\ldots,1$, starting with $\lambda=\Lambda$.
At an unselected position, record
$\langle\alpha_{i_e}^\vee,\lambda\rangle$ and leave
$\lambda$ unchanged.
At a selected position, record zero and replace $\lambda$ by
\[
 s_{i_e}\lambda
 =\lambda-\langle\alpha_{i_e}^\vee,\lambda\rangle\alpha_{i_e}.
\]
Immediately before position $e$ is processed, one has
$\lambda=\sigma_{e,k}\Lambda$, which proves the algorithm.
All coordinates beyond $k$ are zero.

Next, list
$V_k:=\{u\in V_{\bw,v}:q_u\leq k\}=\{u_1<\cdots<u_s\}$,
put $d_a=(\bdelta_k(\Lambda))_{q_{u_a}}$, and read the local
cycle weights $c_{ab}=\gamma_{u_a}(e_{u_b,\mathrm{nw}})$.
Then
\[
 m_{u_a,k}(\Lambda)=d_a-\sum_{b>a}c_{ab}d_b.
\]
Thus, once the local cycle weights are known, the
multiplicities are obtained directly from the Weyl-group
calculation of the parameter vector.
\end{remark}

\subsection{Examples}\label{subsec:examples}

\begin{example}[A factorization in type $A_3$]\label{ex:A3-factorization}

Let
\[
 \bw=(1,3,2,1),\qquad v=s_2,
 \qquad \mathbf h=(1,2,3,2,1).
\]
The word $\mathbf h$ is reduced for $w_0s_2$, and the leftmost
reduced subexpression of $v$ in $\bw$ selects only position $3$.
Thus $J_{\bw,v}=\{1,2,4\}$.
Label the trivalent vertices $u_1<u_2<u_3$, so that
$(q_{u_1},q_{u_2},q_{u_3})=(1,2,4)$, and abbreviate
$M_a=M_{u_a}$ and $\eta_a=\eta_{u_a}^{\bw}$.

We first compute the cycle weights directly from the weave.
In the braid calculations below, a string such as $12321$
denotes the word $(1,2,3,2,1)$.
At the first step, contracting the final $11$ in $123211$
creates $u_1$; its cycle has weights $(0,0,0,0,1)$ on the
resulting word $12321$.
Before the second contraction, use
\[
 12321\longrightarrow13231\longrightarrow13213.
\]
The weights of $\gamma_{u_1}$ become $(0,0,0,1,0)$, so
$\gamma_{u_1}(e_{u_2,\mathrm{nw}})=0$.
Contracting the appended $33$ creates $u_2$.
The third step appends $2$ and is an ascent, giving $132132$.
On this word, the weights of $\gamma_{u_1}$ and $\gamma_{u_2}$
are $(0,0,0,1,0,0)$ and $(0,0,0,0,1,0)$, respectively.
For the fourth step, use
\[
 132132\longrightarrow132312\longrightarrow123212
 \longrightarrow123121.
\]
The corresponding cycle weights are
\[
\begin{array}{c|c|c}
 \text{word} & \gamma_{u_1} & \gamma_{u_2}\\ \hline
 132132 &(0,0,0,1,0,0)&(0,0,0,0,1,0)\\
 132312 &(0,0,0,0,1,0)&(0,0,0,1,0,0)\\
 123212 &(0,0,0,0,1,0)&(0,1,0,0,0,0)\\
 123121 &(0,0,0,1,0,1)&(0,1,0,0,0,0)
\end{array}
\]
Appending $1$ and contracting the final $11$ creates $u_3$.
Thus the northwest incoming weights at $u_3$ are $1$ and $0$.
Consequently,
\[
 C=\begin{pmatrix}0&0&1\\0&0&0\\0&0&0\end{pmatrix},
 \qquad
 \eta_1=E_1,\qquad
 \eta_2=E_2,\qquad
 \eta_3=E_1+E_4.
\]
This calculation uses only the local cycle rules.
Proposition~\ref{prop:weave-parameters} identifies these vectors
with the Lusztig parameters of $M_1,M_2,M_3$.

Next, the only selected reflection is $s_2$ at position $3$.
For $k=4$, the coroots in
\eqref{eq:prefix-delta-coordinates} are
\[
 \sigma_{1,4}^{-1}\alpha_1^\vee
   =\alpha_1^\vee+\alpha_2^\vee,
 \qquad
 \sigma_{2,4}^{-1}\alpha_3^\vee
   =\alpha_2^\vee+\alpha_3^\vee,
 \qquad
 \sigma_{4,4}^{-1}\alpha_1^\vee=\alpha_1^\vee.
\]
The full prefix calculation is therefore
\[
\begin{array}{c|ccc}
 k & \bdelta_{k,1}&\bdelta_{k,2}&\bdelta_{k,3}\\ \hline
 1&E_1&0&0\\
 2&E_1&0&E_2\\
 3&E_1&E_1+E_2&E_2\\
 4&E_1+E_4&E_1+E_2&E_2
\end{array}
\]
A self-dual simple module with parameter
$\mathbf d=(d_1,d_2,0,d_4)\in\ZZ_{\geq0}^4$ is a monomial
in this seed exactly when $d_1\geq d_4$.
Its multiplicities are then
\[
 (m_1,m_2,m_3)=(d_1-d_4,d_2,d_4).
\]
In particular, writing
$X_j=M(w\varpi_j,s_2\varpi_j)$, we obtain
\[
 X_1\simeq M_3,\qquad
 X_2\simeq M_1\odot M_2,\qquad
 X_3\simeq M_2.
\]
The middle identity is a nontrivial factorization, while the
first illustrates the subtraction of a cycle contribution in
recovering the multiplicities.
Conversely, the self-dual simple module $Y\in\mathscr C_w$
with parameter $E_4$ satisfies the support condition but has
\[
 (I_3-C)\pi_J\bigl(\ba^{\bw}(Y)\bigr)=(-1,0,1)^{\mathsf T}.
\]
Thus $Y$ is not a monomial in this seed: the support condition
alone does not suffice.
\end{example}

\begin{example}[Multiplicity two in type $D_4$]\label{ex:D4-multiplicity}
Number the vertices so that $2$ is adjacent to $1,3,4$.
Set
\[
 c=s_1s_3s_4s_2,\qquad v=c,\qquad
 w=s_2c,\qquad \bw=(2,1,3,4,2).
\]
The word $\bw$ is reduced. Since $c^3=w_0$ is represented
by the reduced word $(1,3,4,2)^3$, we may choose
\[
 \mathbf h=(1,3,4,2,1,3,4,2)
\]
as a reduced word for $w_0v^{-1}=c^2$.
The leftmost reduced subexpression for $v$ selects positions
$2,3,4,5$, so $J_{\bw,v}=\{1\}$.
Thus there is a single trivalent vertex $u$, with
$q_u=1$ and $\eta_u^{\bw}=E_1$.
Consequently, $M_u=L(2)$, where $L(2)$ is the self-dual
simple module over $R(\alpha_2)$.

The cycle multiplicity can be computed without any braid
moves. The first letter of $\bw$ contracts with the last
letter of $\mathbf h$, creating $u$.
The remaining letters are ascents and complete
$\mathbf h(1,3,4,2)$ to a reduced expression of $w_0$.
On the final slice, $\gamma_u$ has weight $1$ on its eighth
edge and weight $0$ elsewhere. Therefore
\[
 \gamma_{u,5}
 =s_2s_4s_3s_1(\alpha_2^\vee)
 =\alpha_1^\vee+2\alpha_2^\vee
   +\alpha_3^\vee+\alpha_4^\vee.
\]
Indeed, the first three reflections, applied from right to
left, add the three outer coroots to $\alpha_2^\vee$;
applying $s_2$ then gives the displayed highest coroot.
Hence
\[
 (m_{u,5}^{\,1},m_{u,5}^{\,2},m_{u,5}^{\,3},m_{u,5}^{\,4})
 =(1,2,1,1).
\]
Equivalently, the parameter formula gives
$\bdelta_{5,2}=2E_1$ and $\bdelta_{5,j}=E_1$ for
$j\in\{1,3,4\}$.
The factorization theorem yields
\[
 M(w\varpi_2,v\varpi_2)\simeq L(2)^{\odot2},
 \qquad
 M(w\varpi_j,v\varpi_j)\simeq L(2)
 \quad(j=1,3,4).
\]
Thus a fundamental weight already produces multiplicity
$2$, arising from the coefficient of the central coroot
in the highest coroot of $D_4$.
\end{example}

\end{document}